\documentclass[11pt, oneside, 
a4paper]{article}
\usepackage[ansinew]{inputenc}
\usepackage[english]{babel}
\usepackage{amsbsy,amscd,amsfonts,amsmath,amsopn,amssymb,amstext,amsthm,amsxtra,array,color,fancybox,float,graphicx,latexsym,subfigure,srcltx,times,tikz,url}
\newcommand{\hip}{\{x_n=0\}}
\newcommand{\hipp}{\{x_3=0\}}

\newcommand{\norma}[1]{M_{#1}^{nor}}
\newcommand{\mx}[1]{W_n}
\newcommand{\dom}{\rrr^{n^2-n}_{\le0}}
\newcommand{\spann}{\operatorname{span}}
\newcommand{\rrr}{{\overline{\R}}}
\newcommand{\mayor}{{_{\ge0}}}

\newcommand{\mr}{\operatorname{mr}}
\newcommand{\MR}{\operatorname{MR}}
\newcommand{\mc}{\operatorname{mc}}
\newcommand{\MC}{\operatorname{MC}}

\newcommand{\sucen}[2]{{#1_{1},\hdots,#1_{#2}}}

\newcommand{\realamp}{\R\cup\{-\infty\}}

\newcommand{\m}{\medskip}

\newcommand{\N}{\mathbb{N}}

\newcommand{\R}{\mathbb{R}}

\newcommand{\HH}{\mathcal{H}}

\newcommand{\PP}{\mathcal{P}}

\newcommand{\id}{\operatorname{id}}
\newcommand{\tr}{\operatorname{tr}}

\renewcommand{\int}{\operatorname{int}}

\newcommand{\row}{\operatorname{row}}

\newcommand{\col}{\operatorname{col}}

\newcommand{\card}{\operatorname{card}}

\newcommand{\diag}{\operatorname{diag}}

\newtheorem{thm}{Theorem}

\newtheorem{dfn}[thm]{Definition}
\newtheorem{cor}[thm]{Corollary}
\newtheorem{prop}[thm]{Proposition}
\newtheorem{ex}[thm]{Example}

\newtheorem{alg}[thm]{Algorithm}

\title{Matrices commuting with a given normal  tropical matrix}
\author{J. Linde\\Dpto. de Algebra\\Facultad de Matem\'{a}ticas\\Universidad Complutense\\\texttt{jorgelinde@ucm.es}
\and
M.J. de la Puente\thanks{
Corresponding author}\\Dpto. de Algebra\\Facultad de Matem\'{a}ticas\\Universidad Complutense\\\texttt{mpuente@mat.ucm.es}\\Phone: 34--91--3944659}

\date{} 

\begin{document}
\maketitle


\begin{abstract}
Consider the space $\norma n$ of   square  normal matrices $X=(x_{ij})$ over $\realamp$, i.e., $-\infty\le x_{ij}\le0$ and $x_{ii}=0$. Endow $\norma n$ with the tropical sum $\oplus$ and  multiplication $\odot$.
Fix a real matrix $A\in\norma n$ and consider the set $\Omega(A)$ of matrices in $\norma n$ which commute with $A$.
We prove that  $\Omega(A)$ is a finite union of alcoved  polytopes; in particular, $\Omega(A)$ is a finite union of convex sets.
The set $\Omega^A(A)$ of  $X$ such that  $A\odot X=X\odot A=A$ is also a finite union of alcoved  polytopes. The same is true for the set $\Omega'(A)$ of  $X$ such that  $A\odot X=X\odot A=X$.

A topology is given to $\norma n$. Then, the set $\Omega^{A}(A)$ is a neighborhood of the identity matrix $I$.  If $A$ is strictly normal, then $\Omega'(A)$ is a neighborhood of the zero matrix.
In one case,   $\Omega(A)$ is a neighborhood of $A$. We give an upper bound for the dimension of $\Omega'(A)$.   We explore the relationship between the polyhedral complexes $\spann A$, $\spann X$ and $\spann (AX)$, when $A$ and $X$ commute. Two matrices, denoted $\underline{A}$  and $\overline{A}$, arise from $A$, in connection with $\Omega(A)$. The geometric meaning  of them   is given in detail, for one example. We produce examples of matrices which commute, in any dimension.
\end{abstract}

\section{Introduction}\label{sec:intro}

Let  $n\in\N$ and  $K$ be a field. Fix a matrix $A\in M_n(K)$ and consider $K[A]$, the algebra of polynomial expressions in $A$.  In classical mathematics, the set $\Omega(A)$ of matrices   commuting with $A$ is well--known: $\Omega(A)$ equals $K[A]$  if and only if the characteristic and minimal polynomials of $A$ coincide. Otherwise, $K[A]$ is a proper linear subspace of $\Omega(A)$; see \cite{Prasolov}, chap. VII.

\m
In this paper we study the analogous of $\Omega(A)$ in the tropical setting. Moreover,  we restrict ourselves to square  \emph{normal} matrices over $\overline{\R}:=\realamp$, i.e., matrices $A=(a_{ij})$  with $a_{ii}=0$ and $-\infty\le a_{ij}\le0$, for all $i,j$. The set of all such matrices, endowed with the tropical operations  $\oplus=\max$ and  $\odot=+$, is denoted $\norma n$.

\m
For any $r\in \R_{\le0}$, the half--line $[-\infty,r):=\{x:-\infty\le x<r\}$ is open in $\rrr_{\le0}$ with the usual interval topology. A Cartesian product of such half--lines is open in $\dom$ with the usual product topology. \label{dfn:topology}
The half--line $(r,0]:=\{x:r<x\le0\}$ is open in $\rrr_{\le0}$. A Cartesian product of such half--lines is open in $\dom$.

\m
The set $\norma n$ can be  identified with $\dom$ and,
via this identification, $\norma n$ gets a \emph{topology.}
Consider a matrix $X\in\norma n$ and a subset $V\subseteq \norma n$. We say that $V$ is a \emph{neighborhood} of $X$ if there exists an open subset $U\subseteq \norma n$ such that $X\in U\subseteq V$ (we do not require $V$ to be open).

\m
Let  $\Omega(A)$ be  the subset  of matrices  commuting with a given  real matrix $A$, i.e., $X\in \norma n$  such that  $A\odot X=X\odot A$. The tropical analog of $K[A]$ inside $\norma n$ is the set $\PP(A)$ of tropical powers of $A$. In general, $\Omega(A)$ is larger than $\PP(A)$ (see proposition \ref{prop:Eijeps}). \label{commment:larger}

\m
Our new results are gathered in sections \ref{sec:general}, \ref{sec:perturbations} and \ref{sec:geom}.
In section \ref{sec:general} we prove that $$\Omega(A)=\bigcup_w \Omega_w(A)$$ is  a finite union of alcoved  polytopes,
  (see corollary \ref{cor:finite_un}). In particular, $\Omega(A)$ is a finite union of convex sets.

\m
 Two important subsets of $\Omega(A)$ are
 $$\Omega^A(A)=\{X\in \Omega(A):X\odot A=A\odot X=A\}$$
 and
 $$\Omega'(A)=\{X\in \Omega(A):X\odot A=A\odot X=X\}.$$
  Both are finite unions of alcoved  polytopes (see theorems \ref{thm:alcov^A} and  \ref{thm:alcov'}). Moreover, $\Omega^A(A)$ is a neighborhood (not necessarily open) of the identity matrix $I$. If $A$ is strictly normal, then $\Omega'(A)$ is a neighborhood of the zero matrix $0$ (see propositions \ref{prop:neighI} and \ref{prop:neigh0}).

\m
The study of $\Omega^{A}(A)$ and $\Omega'(A)$ lead us to two matrices arising  from $A$, denoted $\underline{A}$ and $\overline{A}$, and we prove
$$\underline{A}\le A\le \overline{A},$$  (see proposition \ref{prop:ineq_bars}).  Moreover, $X\le \underline{A}$ is a necessary condition for  $A\odot X=X\odot A=A$, and $\overline{A}\le X$ is a necessary condition for  $A\odot X=X\odot A=X$ (see corollary \ref{cor:neces_cond}).
This provides an upper bound for the dimension of  $\Omega'(A)$ (see corollary \ref{cor:dimL'}).
The matrix $\underline{A}$ is explicitly given in expression (\ref{eqn:underl}), while the definition and computation of  $\overline{A}$ is more involved  (see definition \ref{dfn:overl}). 

\m
In section \ref{sec:perturbations} we study some instances of commutativity of matrices under perturbations. Theorem \ref{thm:perturba_convexo} is an easy way to produce two real matrices in $\norma n$ which commute. Another way to obtain two such matrices is given in theorem \ref{thm:perturba_no_convexo}.  The geometry is different in both instances: 
in the first case, the polyhedral complexes (i.e., tropical column spans) associated to the matrices are convex, but not so in the second.
Under certain hypothesis we prove that  $\Omega(A)$ is a neighborhood of $A$ (see corollary \ref{cor:nbhd}).

\m
Section \ref{sec:geom} has an exploratory nature.  We examine the relationship among the complexes $\spann A$, $\spann B$, $\spann (AB)$ and $\spann (BA)$ when commutativity is present or absent. In addition,
the geometric meaning  of the  matrices  $\underline{A}, A$ and $\overline{A}$ is given in full detail, for one example  in the paper.
We believe that classical convexity of  $\spann A$ depends on the matrices $\underline{A}$ and $\overline{A}$. We suspect that  this is related to the question of commutativity.
We leave two open questions in pages p. \pageref{quest:perturbation} and \pageref{quest:generalize}.

\m
Alcoved polytopes play a crucial role in this paper. By definition, a polytope $\PP$ in $\R^{n-1}$ is  \emph{alcoved} if it can be described by inequalities $c_{i}\le x_i\le b_{i}$ and $c_{ik}\le x_i-x_k\le b_{ik}$ , for some $i,k\in[n-1]$, $i\neq k$, and $c_i, b_i,c_{ik}, b_{ik}\in\R\cup\{\pm\infty\}$. They are classically convex sets.  Alcoved polytopes  have been studied in \cite{Lam_Postnikov,Lam_PostnikovII}. In connection with tropical mathematics,  they appeared in
\cite{Jimenez_P,Joswig_Kulas,Joswig_Sturm_Yu,Puente_kleene,Werner_Yu}.
\emph{Kleene stars} are matrices $A$ such that $A={A}^*$, where $*$ is the so--called Kleene operator. Alcoved polytopes and Kleene stars are closely related notions; see \cite{Puente_kleene,Sergeev_def,Sergeev_S_B}.

\m
By definition, a matrix $A=(a_{ij})$ over $\rrr$ is \emph{normal} if $a_{ii}=0$ and $-\infty\le a_{ij}\le0$, for all $i,j$. It is \emph{strictly normal} if, in addition,  $-\infty\le a_{ij}<0$, for all $i\neq j$.
There are FOUR REASONS for us to restrict to normal matrices.  First, it is not all too restrictive. Indeed,
 by the Hungarian Method (see \cite{Butkovic,Butkovic_S,Kuhn,Papa}), for every matrix $A$ there exist a (not unique) similar  matrix   $N$ which is normal. In practice, this means that by a relabeling of the columns of $A$ and a translation, any $A$ can be assumed to be normal. Second, normality of $A$ has a clear geometric meaning in $\R^{n-1}$. Consider the alcoved polytope
\begin{equation}\label{eqn:CA}
C_A:=\left\{x\in\R^{n-1}:\  {{a_{in}\le x_i\le-a_{ni}}\atop {a_{ik}\le x_i-x_k\le-a_{ki}}};\  1\le i\neq k\le n-1\right\}.
\end{equation}
Then, $A$ is normal if and only if the zero vector belongs to $C_A$
and the columns of the matrix $A_0$ (see definition in p. \pageref{not:sub_cero}), viewed as points in $\R^{n-1}$, lie around the zero vector and are listed in a predetermined order (and this order is a kind of orientation in $\R^{n-1}$); see \cite{Puente_kleene} and also
\cite{Johns_Kambi_Idempot,Izha_Johns_Kambi_Proj,Izha_Johns_Kambi_Groups}. Third, when computing examples, normal matrices are easy to handle, due to  inequalities (\ref{eqn:chain}).
Fourth and last, normal matrices satisfy many max--plus properties (e.g., they are strongly definite; see \cite{Butkovic_S,Butkovic_libro}).

\m
Some aspects of commutativity in tropical algebra (also called max--plus algebra or max--algebra) have been addressed earlier. It is known that two commuting matrices have a common eigenvector; see \cite{Butkovic_libro}, sections 4.7, 5.3.5 and 9.2.2.   In \cite{Katz_Schn_Sergeev} it is proved that the critical digraphs of two commuting irreducible matrices have a common node.

\section{Background and notations}\label{sec:background}
For $n\in\N$, set $[n]:=\{1,2,\ldots,n\}$. Let  $\R_{\le0}$, $\R_{\ge0}$, $\overline{\R}_{\le0}$, etc.  have the obvious meaning. On  $\overline{\R}_{\le0}$, i.e., on the closed unbounded half--line $[-\infty,0]$, we consider the \emph{interval topology}: an open set in $[-\infty,0]$ is either a finite intersection or an arbitrary union of sets of the form $[-\infty,a)$ or $(b,0]$, with $-\infty<a,b<0$.

\m
$\oplus=\max$ is the tropical sum and $\odot=+$ is the tropical product. For instance, $3\oplus (-7)=3$  and $3\odot (-7)=-4$.
Define tropical sum and product of matrices following the same rules of classical linear algebra, but replacing addition (multiplication) by tropical addition (multiplication). Consider order $n$ square matrices. The \emph{tropical multiplicative identity} is $I=(\alpha_{ij})$, with  $\alpha_{ii}=0$ and $\alpha_{ij}=-\infty$, for $i\neq j$. The  \emph{zero matrix} is denoted $0$ (every entry of it is null).
We will never use classical  multiplication of matrices; thus $A\odot X$ will be written $AX$,  for matrices $A,X$, for simplicity.

\m
If $A=(a_{ij})$ and $B=(b_{ij})$ are matrices of the same order, then $A\le B$ means $a_{ij}\le b_{ij}$, for all $i,j$.

\m
By definition, a square matrix $A=(a_{ij})$ over $\rrr$ is \emph{normal} if $a_{ii}=0$ and $-\infty\le a_{ij}\le0$, for all $i,j$. Thus, $A$ is normal if and only if $I\le A\le0$. Let us define $A^0$ to be the identity matrix $I$.
So we have
\begin{equation}\label{eqn:chain}
I=A^0\le A\le A^2\le A^3\le\cdots\le  0
\end{equation} since tropical multiplication by any matrix is monotonic (because it amounts to computing certain sums and maxima). By a theorem of Yoeli's (see \cite{Yoeli}), we have $A^{n-1}=A^{n}=A^{n+1}=\cdots$ and we denote this matrix by $A^*$ and call it the \emph{Kleene star}  of $A$. A matrix $A$ is a Kleene star if $A=A^*$.

\m
A normal matrix $A$ is \emph{strictly normal} if $a_{ij}<0$, whenever $i\neq j$.

\m
Let $\norma n$ denote the family of order $n$ normal matrices over $\rrr$. It is in bijective correspondence  with $\dom$. We consider the \emph{product interval topology} on $\dom$. The bijection carries this topology onto  $\norma n$. The \emph{border} of $\norma n$
is the set of matrices $A$ such that  $a_{ij}=0$ or $-\infty$, for some $i\neq j$.

\m
We will write the coordinates of points in $\R^n$ in columns.
Let $A\in \R^{n\times m}$ and denote by $a_1,\ldots,a_m\in \R^n$ the columns of $A$.
The \emph{(tropical column) span} of $A$ is, by definition,
\begin{eqnarray}
\spann A:&=\{(\mu_1\odot a_1)\oplus\cdots\oplus (\mu_m\odot a_m) \in \R^n: \mu_1,\ldots,\mu_m\in\R\}\\
\nonumber&=\max\{\mu_1u+a_1,\ldots,\mu_mu+a_m: \mu_1,\ldots,\mu_m\in\R\}
\end{eqnarray}
where $u=(1,\ldots,1)^t$ and maxima are computed coordinatewise.
We will never use classical linear spans in this paper.
Clearly, the set $\spann A$ is closed under classical addition  of the vector $\mu u$, for $\mu\in\R$, since $\odot=+$. Therefore, the hyperplane section  $\hip\cap\spann A$ determines $\spann A$ completely. The set $\hip\cap\spann A$ is a connected polyhedral complex of impure dimension  $\le n-1$ and  it is not convex, in general. Let $A$ be normal. Then  $\spann A=C_A$ in (\ref{eqn:CA}) (and so it is convex) if and only if  $A$ is a Kleene--star; see \cite{Puente_kleene,Sergeev_def}. Throughout the paper, we will identify the hyperplane $\hip$ inside $\R^n$ with $\R^{n-1}$. In particular, columns of order $n$ matrices having zero last row are considered as points in $\R^{n-1}$.

\m
For any $d\in \R^n$, $\diag d$ denotes the square matrix whose diagonal is $d$ and is $-\infty$ elsewhere.

\m
For any real matrix $A$, the matrix $A_0$ is defined as the tropical product
\begin{equation}
A\diag(-\row(A,n)).
\end{equation}
Thus, the $j$--th column of $A_0$ is a tropical multiple  of the corresponding column of $A$  (i.e., the $j$--th column of $A_0$ is the sum of the vector $-a_{nj} u$ and the $j$--th column of $A$). The  last row of $A_0$ is zero. \label{not:sub_cero} Therefore, the matrix $A_0$ is used to draw the complex $\{x_n=0\}\cap \spann A$ inside $\R^{n-1}$. The sets $\spann A$ and $\{x_n=0\}\cap \spann A$ determine each other.

\m
The simplest objects in the tropical plane $\overline{\R}^2$ are lines. Given
 a tropical linear form
$$p_1\odot X\oplus p_2\odot Y\oplus p_3=\max\{p_1+X ,
    p_2+Y,p_3
     \}$$ a \emph{tropical line} consists  of the points  $(x,y)^t$ where this \emph{maximum is attained, at least, twice}. Such twice--attained--maximum condition is the tropical analog of the classical vanishing point set. Denote this line by $L_p$, where $p=(p_1,p_2,p_3)\in\R^3$.
Lines in the  tropical plane are \emph{tripods}. Indeed,  $L_p$ is  the union of three rays meeting at point $(p_3-p_1,p_3-p_2)^t$, in the directions  west,   south and north--east. The point is  called the \emph{vertex} of $L_p$.

Take $p=0$. The line $L_0$ splits the plane  $\overline{\R}^2$ into three closed sectors $S_1:=\{x\ge0,\  x\ge y\}$, $S_2:=\{x\le y,\ y\ge 0\}$ and $S_3:=\{x\le0,\ y\le0\}$. An order 3 real matrix $A$ is normal if and only if  (omitting the last row in $A_0$, which is zero) each column of $A_0$ lies in the corresponding sector i.e.,  $\col (A_0,j)\in S_j$, for $j=1,2,3$.  For instance,  consider the normal matrix $B$ and take $B_0$ in example \ref{ex:de_nuevo}, figure \ref{fig_03} top centre, p. \pageref{ex:de_nuevo}. Notice that $(5,1)^t\in S_1$, $(-3,0)^t\in S_2$ and
$(-1,-6)^t\in S_3$.  An analogous statement holds for $\overline{\R}^{n-1}$ and order $n$ matrices. See
\cite{Brugalle_fran,Brugalle_engl,Gathmann,Franceses,Rusos,Litvinov_ed,Mikhalkin_W,Richter} for an introduction to tropical geometry. See
\cite{Akian_HB,Baccelli,Butkovic,Butkovic_libro,Cuninghame,Cuninghame_New,Wagneur_M,Zimmermann_K} for an introduction to tropical (or max--plus) algebra.

\section{Normal matrices which commute with $A$}\label{sec:general}
The set $\norma 2$ is commutative, since $AB=BA=A\oplus B$, for any $A,B\in\norma 2$.
Thus, we will study the set
\begin{equation}
\Omega(A):=\{X\in \norma n: AX=XA\},
\end{equation}
for a real matrix $A\in \norma n$ and $n\ge3$.

\m
If $A\in\norma n$ is real and $\lambda\in\R$, then $\lambda\odot A=\lambda u+A$  is normal if and only if $\lambda=0$, where $u$ denotes the order $n$ one matrix. Together with (\ref{eqn:chain}), this means that the tropical analog of $K[A]$ inside $\norma n$ is the set of powers of $A$ together with the zero matrix
\begin{equation}\label{eqn:PA}
\PP(A):=\{I=A^0, A, A^2, \ldots, A^{n-1}=A^*,0\}.
\end{equation}

For $A\in\norma n$ real, set
\begin{equation}\label{eqn:m(A)M(A)}
m(A):=\min_{i,j\in [n]} a_{ij}=\min_{i\neq j\in [n]} a_{ij}\in\R_{\le0}, \qquad M(A):=\max_{i\neq j\in [n]} a_{ij}\in\R_{\le0}.
\end{equation}

\m
For each $r\in \R$, and  $i,j\in [n]$, $i\neq j$, let $E_{ij}(r)\in\norma n$ denote the matrix whose $(i,j)$ entry equals $r$, being  zero everywhere else. For a generic $A\in\norma n$ the matrix $E_{ij}(r)$ is not a power of $A$.

\m
The following proposition  shows that, in general, $\Omega(A)$ is larger than $\PP(A)$.

\begin{prop}\label{prop:Eijeps}
For any real $A\in\norma n$   there exist  $\epsilon>0$ and $i,j\in [n]$ with  $i\neq j$ such that  $E_{ij}(-\epsilon)\in\Omega(A)$.
\end{prop}
\begin{proof}
Fix $i$, $j$ and $\epsilon$. We have $AE_{ij}(-\epsilon)=E_{ij}(\alpha)$ and $E_{ij}(-\epsilon)A=E_{ij}(\beta)$, where

$\alpha=\max\{ a_{i1},\ldots, a_{i,i-1},-\epsilon,a_{i,i+1},\ldots,a_{in}\}$ and

$\beta=\max\{a_{1j},\ldots, a_{j-1,j},-\epsilon, a_{j+1,j},\ldots,a_{nj}\}$.

If   $a_{ij}=0$, then  $\alpha=\beta=a_{ij}=0$, whence $AE_{ij}(-\epsilon)=E_{ij}(-\epsilon)A=0$.

Assume now that $A$ is strictly normal. Then $M(A)<0$.
For any $\epsilon$ with $M(A)<-\epsilon<0$ and any $i\neq j$, we have $\alpha=\beta=-\epsilon$, whence $AE_{ij}(-\epsilon)=E_{ij}(-\epsilon)A=E_{ij}(-\epsilon)$.
\end{proof}

\m
Let $\mx n$  be the set of empty--diagonal order $n$ matrices with entries in $[n]^2$ (the diagonal is irrelevant in these matrices).  Each $w\in \mx n$ is called a \emph{winning position} or a \emph{winner}. Set
\begin{align}\label{eqn:LwA}
\Omega_w(A):=\{X\in \Omega(A): &(AX)_{ij}=a_{i,w(i,j)_1}+x_{w(i,j)_1,j}=\notag\\
&(XA)_{ij}=x_{i,w(i,j)_2}+a_{w(i,j)_2,j}, \text{\ for\ } i,j\in[n], i\neq j\}.
\end{align}

\begin{ex}\label{ex:1}
Consider
$$A=\left[\begin{array}{rrrr}
   0&-4&-6&-3\\-6&0&-4&-3\\-3&-6&0&-3\\-6&-3&-3&0
   \end{array}\right],\quad B=\left[\begin{array}{rrrr}
   0&-4&-4&-6\\-2&0&-3&-4\\-5&-6&0&-5\\-6&-5&-2&0
   \end{array}\right].$$
   Then
   $$AB=BA=\left[\begin{array}{rrrr}
   0&-4&-4&-3\\-2&0&-3&-3\\-3&-6&0&-3\\-5&-3&-2&0
   \end{array}\right]$$
   so that $B\in \Omega_w(A)$ with
   $$w=\left[\begin{array}{cccc}
   &(1,1)&(1,3)&(4,1)\\
   (2,1)&&(2,3)&(4,2)\\
   (1,3)&(2,2)&&(4,3)\\
   (2,3)&(2,4)&(4,3)&
   \end{array}\right].\qed$$
\end{ex}

\begin{ex} \label{ex:2}
For any real $A\in\norma n$,
\begin{itemize}
\item if  $\tr$ denotes the transposition operator, then $I\in \Omega_{\tr}(A)$,
\item if $\id$ denotes the identity operator, then $0, A^*\in \Omega_{\id}(A)$.
\end{itemize}
\end{ex}

\begin{prop}\label{prop:alcov_w}
For any real $A\in\norma n$, $\Omega_w(A)$ is an alcoved polytope.
\end{prop}
\begin{proof}
Fix $i,j\in[n], i\neq j$. Then (\ref{eqn:LwA}) means that
\begin{equation}\label{eqn:eqn}
a_{i,w(i,j)_1}+x_{w(i,j)_1,j}=x_{i,w(i,j)_2}+a_{w(i,j)_2,j}
\end{equation} and the following $2n-2$ inequalities hold
\begin{align}
a_{is}+x_{sj}&\le a_{i,w(i,j)_1}+x_{w(i,j)_1,j}, \text{\ for\ } s\neq w(i,j)_1,\label{eqn:ineq1}\\
x_{it}+a_{tj}&\le x_{i,w(i,j)_2}+a_{w(i,j)_2,j}, \text{\ for\ } t\neq w(i,j)_2.\label{eqn:ineq2}
\end{align}
Equalities and inequalities (\ref{eqn:eqn}), (\ref{eqn:ineq1})  and (\ref{eqn:ineq2}) show that  $X\in \Omega_w(A)$ if and only if $X=(x_{ij})$ belongs to certain alcoved polytope in $\dom\simeq \norma n$.
\end{proof}

Remark 1: Given a winner $w$, if there exist   $i,j,s,t\in [n]$ with $i\neq j$  and  $s\neq t$ such that
\begin{equation}\label{eqn:parallel}
(i,j)\neq (s,t)\neq (j,i),\quad w(i,j)=(s,t),\quad w(s,t)=(i,j),\quad a_{is}+a_{si}\neq a_{jt}+a_{tj},
\end{equation}
 then $\Omega_w(A)$ is empty. Indeed, by (\ref{eqn:eqn}), the following two parallel hyperplanes
 \begin{equation*}
 a_{is}+x_{sj}=x_{it}+a_{tj},\quad a_{si}+x_{it}=x_{sj}+a_{jt},
 \end{equation*} take part in the description of $\Omega_w(A)$.

 For instance,  back to  $A$ in example \ref{ex:1}, if $\tau\in \mx n$ is such that $\tau (1,3)=(2,4)$ and $\tau (2,4)=(1,3)$, then $\Omega_\tau (A)=\emptyset$, because $a_{12}+a_{21}=-10\neq a_{34}+a_{43}=-6$.

Remark 2: Given a winner $w$ and  $i,j\in [n], i\neq j$, if
\begin{equation}\label{eqn:equal_or_sym}
w(i,j)=(i,j) \text{\ or\ } w(i,j)=(j,i),
\end{equation} then equality (\ref{eqn:eqn}) is tautological. In particular,
\begin{equation}\label{eqn:dim}
\dim \Omega_w(A)\le n^2-n-\card  P_w^c,
\end{equation}
where $P_w:=\{(i,j): 1\le i<j\le n \text{\ with \ } w(i,j)=(i,j) \text{\ or\ } w(i,j)=(j,i)\}$ and $^c$ denotes complementary.

\setcounter{thm}{1}
\begin{ex}
(Continued) For $w$, the pairs which do not satisfy (\ref{eqn:equal_or_sym}) are $w(1,2)=(1,1)$, $w(3,2)=(2,2)$ and $w(4,1)=(2,3)$, so that  $P_w^c=\{(1,2), (3,2), (4,1)\}$. It follows that $x_{12}=-4$, $x_{32}=-6$ and $x_{21}=x_{43}$ are some of the equations describing $\Omega_w(A)$. Besides, condition (\ref{eqn:parallel}) is satisfied for no pairs, whence $$0<\dim \Omega_w(A)\le 16-4-3=9.\qed$$
\end{ex}
\setcounter{thm}{4}

Clearly,
\begin{equation}
\Omega(A)=\bigcup_{w\in \mx n} \Omega_w(A)
\end{equation}
and the set $\mx n$ is finite, whence the following corollary is a straightforward consequence of proposition \ref{prop:alcov_w}.

\begin{cor}\label{cor:finite_un}
For any real $A\in\norma n$, $\Omega(A)$ is a finite union of alcoved polytopes. \qed
\end{cor}

\m
The sets $\Omega_w(A)$ are not too natural. On the contrary, the sets $\Omega^S(A)$ described below are more natural but harder to study.
For any $S\in\norma n$, let
\begin{equation}
\Omega^S(A):=\{X\in \Omega(A): XA=AX=S\},
\end{equation}
so that
\begin{equation}
\Omega(A)=\bigcup_{S\in \norma n} \Omega^S(A)
\end{equation}
is a disjoint union.
For instance,  $B\in\Omega^S(A)$, for $S:=BA$ in example \ref{ex:1}.
We also consider the set
\begin{equation}
\Omega'(A):=\{X\in \Omega(A): XA=AX=X\}.
\end{equation}

It is immediate to see that
\begin{enumerate}\label{list:pro}
\item  $A^{j-1}\in \Omega^{A^j}(A)$, for $j\in[n]$. In particular, $I=A^0\in \Omega^{A}(A)$, i.e.,  $AI=IA=A$.
\item  $A^*\in \Omega'(A)$, i.e., $AA^*=A^*A=A^*$.
\item  $0\in \Omega'(A)$, i.e., $A0=0A=0$.
\end{enumerate}

%
%
%

\begin{prop}\label{prop:entre_potencias}
For any  real $A,B\in\norma n$, if that $A^{n-2}\le B\le A^*$, then $B\in \Omega^{A^*}(A)$.
\end{prop}
\begin{proof}
$A^{n-1}=A^n=A^{n+1}=\cdots=A^*$, by Yoeli's theorem, and left or right multiplication by $A$ is monotonic, so that $A^{n-2}\le B\le A^*$ implies $A^*\le AB\le A^*$ and $A^*\le BA\le A^*$.
\end{proof}

Recall $m(A)$ and $M(A)$ defined in (\ref{eqn:m(A)M(A)}). Recall the topology in $\norma n$, defined in p.~\pageref{dfn:topology}.

\m
For $r\in \overline{\R}$, denote by $K(r)=(\alpha_{ij})$ the \emph{constant matrix} such that $\alpha_{ii}=0$ and  $\alpha_{ij}=r$, for all $i\neq j$. For instance, $I=K(-\infty)$ and $0=K(0)$.

\m

\begin{prop}\label{prop:neighI}
For any real $A\in\norma n$, if   $I\le B\le K(m(A))$,  then $B\in \Omega^{A}(A)$. In particular, $\Omega^{A}(A)$ is a neighborhood of the identity matrix $I$.
\end{prop}
\begin{proof}
 The hypothesis $I\le B\le K(m(A))$ means that $B$ is normal and $b_{ij}\le m(A)$, for all $i\neq j$.

 If $i\neq j$, we have
$(AB)_{ij}=\max_{k\in [n]}a_{ik}+b_{kj}=a_{ij}$, since $a_{ik}+b_{kj}\le a_{ik}+m(A)\le m(A)\le a_{ij}$, when $k\neq j$, and $a_{ij}+b_{jj}=a_{ij}$. Similarly, $(BA)_{ij}=a_{ij}$. This shows $AB=BA=A$, so that $B\in\Omega^A(A)$.

The value $m(A)$ defined in (\ref{eqn:m(A)M(A)}) is real. The set $U=\{B: I\le B< K(m(A))\}$ is in bijective correspondence with
 the Cartesian product of  half--lines $[-\infty,m(A))^{n^2-n}$, which is open. Moreover,  $I\in U\subseteq \Omega^A(A)$, proving the neighborhood condition.
\end{proof}

Notice that $m(A)$ equals $-|||A|||$, as defined in \cite{Puente_kleene}.
  There, it is proved that  $|||A|||$ is the \emph{(tropical) radius} of the section $\{x_n=0\}\cap \spann A$, i.e., the maximal tropical distance to the zero vector, from any point on $\{x_n=0\}\cap \spann A$. This conveys a geometrical meaning to proposition \ref{prop:neighI}.\label{note:norm}

\begin{prop}\label{prop:neigh0}
Suppose that  $A\in\norma n$  is real and strictly normal. If $B$ is such that  $K(M(A))\le B\le0$, then $B\in \Omega'(A)$. In particular, $\Omega'(A)$ is a neighborhood of the zero matrix $0$.
\end{prop}
\begin{proof}
We have $M(A)<0$, by strict normality. The hypothesis on $B=(b_{ij})$ means that $M(A)\le b_{ij}$, for  every $i,j\in [n]$ with $i\neq j$.

For  $i\neq j$, we get
$(AB)_{ij}=\max_{k\in [n]}a_{ik}+b_{kj}=b_{ij}$, since $a_{ik}+b_{kj}\le M(A)+b_{kj}\le M(A)\le b_{ij}$, when $k\neq i$, and $a_{ii}+b_{ij}=b_{ij}$. Similarly, $(BA)_{ij}=b_{ij}$. This shows $AB=BA=B$, so that $B\in\Omega'(A)$.

The set $U=\{B: K(M(A))< B\le 0\}$ is in bijective correspondence with
 the Cartesian product of  half--lines $(M(A),0]^{n^2-n}$, which is open. Moreover,  $0\in U\subseteq \Omega'(A)$, proving the neighborhood condition.
\end{proof}

Note that the former proposition is analogous to proposition \ref{prop:neighI}, with the zero matrix  playing  the role of the identity matrix.

\bigskip
Below we describe the sets $\Omega^{A}(A)$ and $\Omega'(A)$ as finite union of alcoved  polytopes. In order to do so, for $i\in[n]$, consider the matrices
\begin{itemize}
\item  $R^i_A =(r^i_{kj})$, with $r^i_{kj}=a_{ij}-a_{ik}$ (difference in $i$--th row;  subscripts  $k,j$ get inverted),
\item $C^i_A =(c^i_{kj})$, with $c^i_{kj}=a_{ki}-a_{ji}$ (difference in $i$--th column; subscripts  $k,j$ don't get inverted).
\end{itemize}
Let $\oplus'$ denote $\min$. Write $R:=\bigoplus'_{i\in[n]}R^i_A$ and $C:=\bigoplus'_{i\in[n]} C^i_A$ and consider
\begin{equation}\label{eqn:underl}
\underline{A}:=R\oplus' C=A\oplus' R\oplus' C,
\end{equation}
the last equality being true since  $r^{i}_{ij}=a_{ij}$ and  $c^{j}_{kj}=a_{kj}$, by normality of $A$. Clearly, $\underline{A}\le A$ and $\underline{A}$ is real and normal, if $A$ is.
\m

Notation: $(\leftarrow,\underline{A}]:=\{X\in\norma n :X\le \underline{A}\}$. This is an alcoved polytope of dimension $n^2-n$.

\begin{thm}\label{thm:alcov^A}
For any real $A\in \norma n$, $\Omega^A(A)$ is a finite  union of alcoved polytopes. Moreover,
$$\Omega_{\tr}(A)\subseteq \Omega^A(A)\subseteq (\leftarrow, \underline{A}].$$
\end{thm}

\begin{proof}
$AX=XA=A$ if and only if
\begin{equation}\label{eqn:alcov^A}
\max_{k\in[n]}a_{ik}+x_{kj}=a_{ij},\quad \max_{k\in[n]}x_{ik}+a_{kj}=a_{ij}, \text{\ for\ } i,j\in [n], i\neq j.
\end{equation}
Now, for each $X=(x_{ij})\in\Omega^A(A)$ there exists some winner $w_X$ such that, for each pair $(i,j)$ with $i\neq j$, the maxima in (\ref{eqn:alcov^A}) are attained at $w_X(i,j)$. Since $W_n$ is finite,
 then (\ref{eqn:alcov^A}) describe a finite  union of alcoved polytopes in the variables $x_{ij}$. Moreover,  $X\le \underline{A}$ follows from (\ref{eqn:underl}) and (\ref{eqn:alcov^A}). In addition,  the maxima in (\ref{eqn:alcov^A}) are attained, at least, for the transposition operator. Therefore, $\Omega_{\tr}(A)\subseteq \Omega^A(A)$.
\end{proof}

\begin{alg} To compute $\underline{A}$, we proceed as follows:  for $1\le i<j \le n$,
\begin{itemize}
\item  compute the minimum and maximum of $\row (A,i)-\row(A,j)$, denoted $\mr_{ij}$ and $\MR_{ij}$, respectively,
\item  compute the minimum and maximum of $\col (A,i)-\col(A,j)$, denoted $\mc_{ij}$ and $\MC_{ij}$, respectively,
\item $\underline{A}_{ij}=\min\{a_{ij}, \mr_{ij}, -\MC_{ij}\}$,
\item $\underline{A}_{ji}=\min\{a_{ji}, -\MR_{ij}, \mc_{ij}\}$.
\end{itemize}
A \emph{sorting algorithm} is needed to compute $\mr_{ij}, \mc_{ij}, \MR_{ij}, \MC_{ij}$. For instance, \emph{Mergesort} has $O(n\log n)$ complexity, whence the complexity of the computation of $\underline{A}$ is $O(n^3\log n)$.
\end{alg}

\begin{ex} \label{ex:3}
For
\begin{equation}\label{eqn:3x3}
B=\left[\begin{array}{rrr}
0&-3&-1\\
-4&0&-6\\
-5&0&0
\end{array}\right] \text{\ we\ get\quad }
\underline{B}=\left[\begin{array}{rrr}
0&-3&-3\\
-5&0&-6\\
-5&-2&0
\end{array}\right].
\end{equation}
On the other hand, for $A$ in example \ref{ex:1},  we get $A=\underline{A}$.\qed
\end{ex}

Notation: $\underline{[A,\rightarrow)}:=\{X\in\norma n :A\le \underline{X}\}$. It is an alcoved polytope, since the definition of $\underline{X}$ involves differences $x_{ij}-x_{kl}$ of two entries.

\m
The proof of the theorem below is similar to the proof of theorem \ref{thm:alcov^A}. Alternatively, theorem \ref{thm:alcov'} is a corollary of theorem \ref{thm:alcov^A}, using that   $X\in \Omega^{A}(A)$ if and only if $A\in \Omega'(X)$.

\begin{thm}\label{thm:alcov'}
For any real $A\in \norma n$, $\Omega'(A)$ is a finite  union of alcoved polytopes. Moreover,
$$\Omega_{\id}(A)\subseteq\Omega'(A)\subseteq\underline{[A,\rightarrow)}.\qed$$
\end{thm}

\m

The sets $(\leftarrow, \underline{A}]$ and $\underline{[A,\rightarrow)}$ are alcoved polytopes, but $\underline{[A,\rightarrow)}$ is trickier than $(\leftarrow, \underline{A}]$. We can compute a \emph{tight description} of any of them, as explained in \cite{Puente_kleene}. It goes as follows.
For any $m\in \N$, any real matrix $H\in\norma m$ yields the alcoved polytope $C_H$ (see (\ref{eqn:CA})),
and it turns out that $C_H=C_{H^*}$. Moreover, the description of this convex set given by $H^*$ is \emph{tight}.

\setcounter{thm}{10}
\begin{ex} (Continued)
Let us compute a tight description of $\underline{[B,\rightarrow)}$, for $B$ in (\ref{eqn:3x3}). The matrix $\underline{X}$ is defined in (\ref{eqn:underl}) and we have $B\le \underline{X}$ if and only if

\begin{alignat*}{2}
-3&\le x_{12}&\qquad\qquad -6&\le x_{23}\\
-3&\le x_{32}-x_{31}&\qquad\qquad -6&\le x_{13}-x_{12}\\
-3&\le x_{13}-x_{23}&\qquad\qquad -6&\le x_{21}-x_{31}\\
-1&\le x_{13}&\qquad\qquad        -5&\le x_{31}\\
-1&\le x_{23}-x_{21}&\qquad\qquad -5&\le x_{21}-x_{23}\\
-1&\le x_{12}-x_{32}&\qquad\qquad -5&\le x_{32}-x_{12}\\
-4&\le x_{21}&\qquad\qquad        0&\le x_{32}\\
-4&\le x_{31}-x_{32}&\qquad\qquad 0&\le x_{12}-x_{13}\\
-4&\le x_{23}-x_{13}&\qquad\qquad 0&\le x_{31}-x_{21}.
\end{alignat*}
Now, in order to write down the matrix $H$,   we perform a \emph{relabeling} of  the unknowns; for instance:
\begin{equation*}
y_{1}=x_{12},\ y_{2}=x_{13},\ y_{3}=x_{21},\ y_{4}=x_{23},\ y_{5}=x_{31},\ y_{6}=x_{32},
\end{equation*}
so that,
\begin{alignat*}{2}\label{eqn:alcov'}
-3&\le y_{1}&\qquad\qquad 0&\le y_{1}-y_{2}\le6\\
-1&\le y_{2}&\qquad\qquad -1&\le y_{1}-y_{6}\le5\\
-4&\le y_{3}&\qquad\qquad -3&\le y_{2}-y_{4}\le4\\
-6&\le y_{4}&\qquad\qquad -5&\le y_{3}-y_{4}\le1\\
-5&\le y_{5}&\qquad\qquad -6&\le y_{3}-y_{5}\le0\\
0&\le y_{6}&\qquad\qquad  -4&\le y_{5}-y_{6}\le3
\end{alignat*}
and we get $\underline{[B,\rightarrow)}=C_H$, with
\begin{equation*}
H=\left[\begin{array}{rrrrrrr}
0&0&-\infty&-\infty&-\infty&-1&-3\\
-6&0&-\infty&-3&-\infty&-\infty&-1\\
-\infty&-\infty&0&-5&-6&-\infty&-4\\
-\infty&-4&-1&0&-\infty&-\infty&-6\\
-\infty&-\infty&0&-\infty&0&-4&-5\\
-5&-\infty&-\infty&-\infty&-3&0&0\\
0&0&0&0&0&0&0
\end{array}\right].
\end{equation*}
Then $H^3=H^4=H^*$, with
\begin{equation*}
H^*=\left[\begin{array}{rrrrrrr}
0&0&-1&-1&-1&-1&-1\\
-1&0&-1&-1&-1&-1&-1\\
-4&-4&0&-4&-4&-4&-4\\
-5&-4&-1&0&-5&-5&-5\\
-4&-4&0&-4&0&-4&-4\\
0&0&0&0&0&0&0\\
0&0&0&0&0&0&0
\end{array}\right]
\end{equation*}
so that  $\underline{[B,\rightarrow)}=C_H=C_{H^*}$, by \cite{Puente_kleene}, and this set is described \emph{tightly} as follows:
\begin{alignat*}{2}
-1&\le y_{1} \le 0&\qquad\qquad -1&\le y_{1}-y_{4}\le 5\\
-1&\le y_{2} \le 0&\qquad\qquad -1&\le y_{1}-y_{5}\le 4\\
-4&\le y_{3} \le 0&\qquad\qquad -1&\le y_{2}-y_{3}\le 4\\
-5&\le y_{4} \le 0&\qquad\qquad -1&\le y_{2}-y_{4}\le 4\\
-4&\le y_{5} \le 0&\qquad\qquad -1&\le y_{2}-y_{5}\le 4\\
0&=y_{6}&\qquad\qquad -4&\le y_{3}-y_{4}\le 1\\
 0&\le y_{1}-y_{2}\le 1&\qquad\qquad -4&\le y_{3}-y_{5}\le 0\\
-1&\le y_{1}-y_{3}\le 4&\qquad\qquad -5&\le y_{4}-y_{5}\le 4.
\end{alignat*}
In particular, $\dim \underline{[B,\rightarrow)}=\dim C_{H^*}=9-3-1=5$.  Undoing the relabeling, we get
\begin{alignat*}{2}
-1&\le x_{12} \le 0&\qquad\qquad -1&\le x_{12}-x_{23}\le 5\\
-1&\le x_{13} \le 0&\qquad\qquad -1&\le x_{12}-x_{31}\le 4\\
-4&\le x_{21} \le 0&\qquad\qquad -1&\le x_{13}-x_{21}\le 4\\
-5&\le x_{23} \le 0&\qquad\qquad -1&\le x_{13}-x_{23}\le 4\\
-4&\le x_{31} \le 0&\qquad\qquad -1&\le x_{13}-x_{31}\le 4\\
0&=x_{32}&\qquad\qquad -4&\le x_{21}-x_{23}\le 1\\
 0&\le x_{12}-x_{13}\le 1&\qquad\qquad -4&\le x_{21}-x_{31}\le 0\\
-1&\le x_{12}-x_{21}\le 4&\qquad\qquad -5&\le x_{23}-x_{31}\le 4.
\end{alignat*}
Write
\begin{equation}\label{eqn:B_overl}
\overline{B}=\left[\begin{array}{rrr}
0&-1&-1\\-4&0&-5\\-4&0&0
\end{array}\right]
\end{equation}
and notice that $\overline{B}\le X$ follows from the first six inequalities above.
\end{ex}
\setcounter{thm}{12}

Computations  as in the former example can be  done for any real matrix $A\in\norma n$, as follows.

\begin{dfn}\label{dfn:relabeling}
For $n\in\N$, a \emph{relabeling} is a bijection  between two sets of variables:  $\{x_{ij}: (i,j)\in [n]^2,\ i\neq j\}$ and $\{y_k: k\in[n^2-n]\}$. By abuse of notation, we  write $y_k=x_{ij}$, for corresponding $y_k$ and $x_{ij}$.
\end{dfn}

\begin{dfn}\label{dfn:overl}
Given $A\in\norma n$ real, suppose that $\underline{[A,\rightarrow)}$ equals $C_{H^*}$, for some idempotent  matrix $H^*=(h^*_{ij})\in \norma  {n^2-n+1}$  and  some relabeling $y_k=x_{ij}$.
Then $\overline{A}=(\alpha_{ij})\in \norma n$,
with $\alpha_{ij}=h^*_{k,n^2-n+1}$, i.e., the entries of $\overline{A}$ are obtained form the last column of $H^*$.
\end{dfn}

The matrix $\overline{A}$ does not depend on the relabeling. The arithmetical complexity of computing $\overline{A}$ is that of $H^*$, which is $O((n^2-n)^3)=O(n^6)$, by the Floyd--Warshall algorithm.

\begin{cor}\label{cor:neces_cond}
For any $A,X\in\norma n$ with $A$ real, $A\le \underline{X}$ implies $\overline{A}\le X$. In particular, $\Omega'(A)\subseteq [\overline{A},\rightarrow)$.
\end{cor}

\begin{proof}
We proceed as in example above and we use theorem \ref{thm:alcov'}.
\end{proof}

\begin{cor}\label{cor:dimL'}
Given $A\in\norma n$ real, suppose that  $\underline{[A,\rightarrow)}$ equals $C_{H^*}$,  for some idempotent  matrix $H^*=(h^*_{ij})\in \norma {n^2-n+1}$. Then
$$\dim \Omega'(A)\le n^2-n-\card Q,$$
 where $Q=\{(i,n^2-n+1): h^*_{i,n^2-n+1}=h^*_{n^2-n+1,i}=0,\text{\ with \ } 1\le i<n^2-n+1\}\cup \{(i,k): h^*_{ik}=h^*_{ki}=0, \text{\ with \ } 1\le i< k\le n^2-n+1\}$.
\end{cor}

\begin{proof}
The description of $\underline{[A,\rightarrow)}$ via $H^*$ is tight, by proposition 2.6 in \cite{Puente_kleene}. Thus, the dimension of $\underline{[A,\rightarrow)}$ drops by one unit each time that a chain of two inequalities in expression (\ref{eqn:CA}) (for $H^*$ instead of $A$), turns into two equalities, which occurs whenever $h_{ik}^*=h_{ki}^*=0$, by normality of $H^*$. Thus, $\dim \underline{[A,\rightarrow)}=n^2-n-\card Q$ and this is an upper bound for $\dim \Omega'(A)$.
\end{proof}

\begin{prop}\label{prop:ineq_bars}
For any $A\in \norma n$ real, we have $\underline{A}\le A\le \overline{A}$.
\end{prop}
\begin{proof}
The inequality $\underline{A}\le A$ was explained in p. \pageref{eqn:underl}. Now consider $X$ such that  $A\le \underline{X}$. Then,
$$A\le \underline{X}\le X,$$ by the same reason,
so that $A\le X$. By definition \ref{dfn:overl}, the matrix $\overline{A}$ is obtained from the last column of $H^*$ and, by \cite{Puente_kleene}, the description of the alcoved polytope $\underline{[A,\rightarrow)}$ as $C_{H^*}$ is tight. Part of this  description is $\overline{A}\le X$. Therefore,   $A\le \overline{A}\le X$, by tightness.
\end{proof}

Some questions arise, such as: 
\begin{enumerate}
\item  We know that $\underline{A}\le A\le \overline{A}$. Does every $X$ with $\underline{A}\le X\le \overline{A}$ commute with $A$? The answer is NO. Example:  take $B$ in (\ref{eqn:3x3}) and
\begin{equation*}
X=\left[\begin{array}{rrr}
0&-2&-2\\-4&0&-5\\-4&0&0\\
\end{array}\right], BX=\overline{B}\neq XB=\left[\begin{array}{rrr}
0&-2&-1\\-4&0&-5\\-4&0&0\\
\end{array}\right].
\end{equation*}

\item We know that $A^*$ and 0 belong to $\Omega'(A)$. Does every $X$ with $A^*\le X\le0$ commute with $A$? The answer is NO. Example: for $B$ in (\ref{eqn:3x3}), we have $B^*=\overline{B}$ in (\ref{eqn:B_overl}) and
\begin{equation*}
X=\left[\begin{array}{rrr}
0&-1&-1\\0&0&-1\\0&0&0\\
\end{array}\right]=XB\neq BX=\left[\begin{array}{rrr}
0&-1&-1\\0&0&-1\\-1&0&0\\
\end{array}\right].
\end{equation*}
\end{enumerate}

\section{Perturbations}

\begin{dfn}\label{dfn:size}
Assume $a,b\in \R\mayor$ with $a\le b$. Then $a,b$ are of the \emph{same size} if $b\le 2a$. Otherwise, $2a<b$ and we say that $a$ is \emph{small with respect to} $b$.
\end{dfn}

\label{sec:perturbations}
In the topological space $\norma n\simeq \overline{\R}^{n^2-n}_{\le0}$ the following is expected to hold true, for any  real matrix $A\in \norma n$:
\begin{enumerate}
\item for $j\in [n]$ and each \emph{sufficiently small perturbation} $X$ of $A^{j-1}$, we  have $AX=XA$, and this is a  perturbation of $A^j$, (including the case that $X$ is a perturbation of $I=A^0$ or of $A^*=A^{n-1}$) \label{item:uno}
\item for each \emph{sufficiently small perturbation} $X$ of $0$, we  have $AX=XA$, and this is a  perturbation of $0$.\label{item:dos}
\end{enumerate}

The point here is, of course, to give a precise meaning of \emph{sufficiently small perturbation}. Although  we are not able to do it yet,  we believe that the statement will be about linear inequalities in terms of the non--zero entries $a_{ij}$ of $A$ and some  perturbing constants $\pm\epsilon_1,\ldots,\pm\epsilon_s$, with $\epsilon_k\ge0$ for $k=1,\ldots,s$, and some $s\ge0$. We further believe that the perturbing constants  must be  \emph{small} with respect to every non--zero absolute value $|a_{ij}|$, according to definition \ref{dfn:size}. Recall that $\Omega(A)$ is larger than $\PP(A)$ (see p.~\pageref{commment:larger}). An intriguing related QUESTION is the following:
is every $X\in \Omega(A)$  a \emph{small perturbation} of some member of $\PP(A)$? \label{quest:perturbation}

\m
Below we present some partial results.

\m
For brevity, write $A\oplus B:=M=(m_{ij})$.

\begin{prop}\label{prop:menor_que_max}
Assume $A,B\in \norma n$ are such that $a_{ik}+b_{kj}\le m_{ij}$, for all $i,j,k\in [n]$. Then $AB=BA=M$. In particular, $B\in \Omega^M(A)$.
\end{prop}
\begin{proof}
By normality, $I\le A\le 0$ and $I\le B\le 0$, whence $A\le AB\le 0$ and $B\le AB\le 0$, since (tropical) left or right multiplication  by any matrix is monotonic.
Thus, $M\le AB$ and, similarly, $M\le BA$ and, by hypothesis, $AB\le M$ and $BA\le M$. Therefore $AB=BA=M$.
\end{proof}

\begin{thm}\label{thm:perturba_convexo}
For each $n\in\N$ and each non positive real number $r$, any two order $n$ matrices $A,B$ having zero diagonal  and all off--diagonal entries in the closed interval $[2r,r]$ satisfy $AB=BA=M$. In particular, $B\in \Omega^M(A)$.
\end{thm}

\begin{proof}
Let  $a_{ii}=b_{ii}=0$ and $2r\le a_{ij}, b_{ij}\le r\le0$, for $i,j\in[n]$.
Fix $i,j\in [n]$ with $i\neq j$. For each $k\in [n]$, we have $a_{ik}+b_{kj}\le 2r\le a_{ij}, b_{ij}$, and we can apply the previous proposition to conclude.
\end{proof}

That is an easy way to produce two real matrices which commute! Moreover,
the matrices $A,B$ and $M$ are idempotent. Indeed, $A\le A^2$ by normality and, since $a_{ij}+a_{jk}\le 2r\le a_{ik}$, we get $A^2\le A$, whence $A={A}^2$; similarly $B=B^2$ and $M=M^2$.
 Here $B\in\Omega(A)$ is a perturbation of $A$ and  $AB=BA=M$ is a perturbation of $A^2=A$, so this is an example of item \ref{item:uno} in p. \pageref{item:uno}, for $j=2$.

\m
In the former theorem, notice that  the absolute value of the entries  $|a_{ij}|$ and $|b_{ij}|$ of $A$ and $B$ are of the \emph{same size}, taken by pairs, as in  definition \ref{dfn:size}.
The reader should compare theorem \ref{thm:perturba_convexo} with example \ref{ex:1}, where $M^2=AB=BA\neq M$, these matrices being different only at entry $(4,1)$. There $A,B$ and $AB$ are idempotent, but $M$ is not.  



\begin{cor}\label{cor:nbhd}
For each $n\in \N$ and  each negative real number $r$, take   $a_{ij}$ in the open interval $(2r,r)$, whenever $i\neq j$ and $a_{ii}=0$, all $i,j\in [n]$. Then $A=(a_{ij})$ is strictly normal and $\Omega(A)$ is a neighborhood of $A$.
\end{cor}

\begin{proof}
The  Cartesian product of intervals $U=(2r,r)^{n^2-n}$ is open  in $\dom$. The image  $U'$ of $U$  in $\norma n$ satisfies $A\in U'\subseteq \Omega(A)$, by theorem \ref{thm:perturba_convexo}, proving the neighborhood condition.
\end{proof}

\m
Corollary \ref{cor:nbhd} is an instance of item \ref{item:uno} in p. \pageref{item:uno}. Below we present another one.

\m
For $n\ge3$, consider $p=(\sucen pn)\in \R^n\mayor$ and $\epsilon \ge0$ and set
\begin{equation}
P(-p,-\epsilon):=\left[\begin{array}{rrrrr}
0&-\epsilon&\cdots&-\epsilon&-p_n\\
-p_1&0&-\epsilon&\cdots&-\epsilon\\
-\epsilon&-p_2&0&\ddots&\vdots\\
\vdots&\ddots&\ddots&\ddots&-\epsilon\\
-\epsilon&\cdots&-\epsilon&-p_{n-1}&0
\end{array}\right]\in\norma n,
\end{equation}
and for $n\ge4$, set
\begin{equation}
Q(-p,-\epsilon):=\left[\begin{array}{rrrrrr}
0&0&\cdots&0&-\epsilon&-p_n\\
-p_1&0&\cdots&\cdots&0&-\epsilon\\
-\epsilon&-p_2&0&\cdots&\cdots&0\\
0&-\epsilon&-p_3&0&\cdots&\vdots\\
\vdots&\ddots&\ddots&\ddots&\ddots&\vdots\\
0&\cdots&0&-\epsilon&-p_{n-1}&0
\end{array}\right]\in\norma n.
\end{equation}

The matrices $P(-p,-\epsilon)$ and $Q(-p,-\epsilon)$ are   perturbations of $P(-p,0)=Q(-p,0)$.

\begin{thm}\label{thm:perturba_no_convexo}
Let $p\in \R^n\mayor$ 
and let $\delta,\epsilon\ge0$ be such that $\delta+\epsilon\le \min_{i\in[n]}p_i$.
Write $m=\min\{\delta,\epsilon\}$.
Then
\begin{enumerate}
\item $P(-p,-\delta)P(-p,-\epsilon)=P(-p,-\epsilon)P(-p,-\delta)=P(-(\delta+\epsilon, \ldots,\delta+\epsilon),-m)$.
\item $Q(-p,-\delta)Q(-p,-\epsilon)=Q(-p,-\epsilon)Q(-p,-\delta)=Q(-(m,\ldots,m),0)$. 
\end{enumerate}
\end{thm}

\begin{proof}
Straightforward computations.
\end{proof}

\begin{ex} \label{ex:4}
Take $p=(4,3,5)$, $\epsilon=1$ and $\delta=2$,
\begin{equation}
P(-p,-2)=\left[\begin{array}{rrr}
0&-2&-5\\-4&0&-2\\-2&-3&0
\end{array}\right], \
P(-p,-1)=\left[\begin{array}{rrr}
0&-1&-5\\-4&0&-1\\-1&-3&0
\end{array}\right].
\end{equation}

By theorem \ref{thm:perturba_no_convexo}, we have
\begin{equation}
P(-p,-2)P(-p,-1)=P(-p,-1)P(-p,-2)=P(-(3,3,3),-1)=\left[\begin{array}{rrr}
0&-1&-3\\-3&0&-1\\-1&-3&0
\end{array}\right].
\end{equation}
Pictures for this example are shown in figure \ref{fig_01}. Write $A=P(-p,-2)$, $B=P(-p,-1)$, $C=AB=BA$. In $\R^2$ we have sketched the intersection of the classical hyperplane $\{x_3=0\}$ with $\spann A, \spann P(-p,0)$ and $\spann B$ on top,  and with $\spann C$ bottom.
To do so, we have used the matrices $A_0$, $P(-p,0)_0$, $B_0$ and $C_0$ as defined in p. \pageref{not:sub_cero}:
\begin{equation*}
A_0=\left[\begin{array}{rrr}
2&1&-5\\-2&3&-2\\0&0&0
\end{array}\right],
P(-p,0)_0=\left[\begin{array}{rrr}
0&3&-5\\-4&3&0\\0&0&0
\end{array}\right],
B_0=\left[\begin{array}{rrr}
1&2&-5\\-3&3&-1\\0&0&0
\end{array}\right],
\end{equation*}
\begin{equation*}
C_0=\left[\begin{array}{rrr}
1&2&-3\\-2&3&-1\\0&0&0
\end{array}\right].
\end{equation*}
\end{ex}

\begin{figure}[h]
 \centering
  \includegraphics[width=13cm]{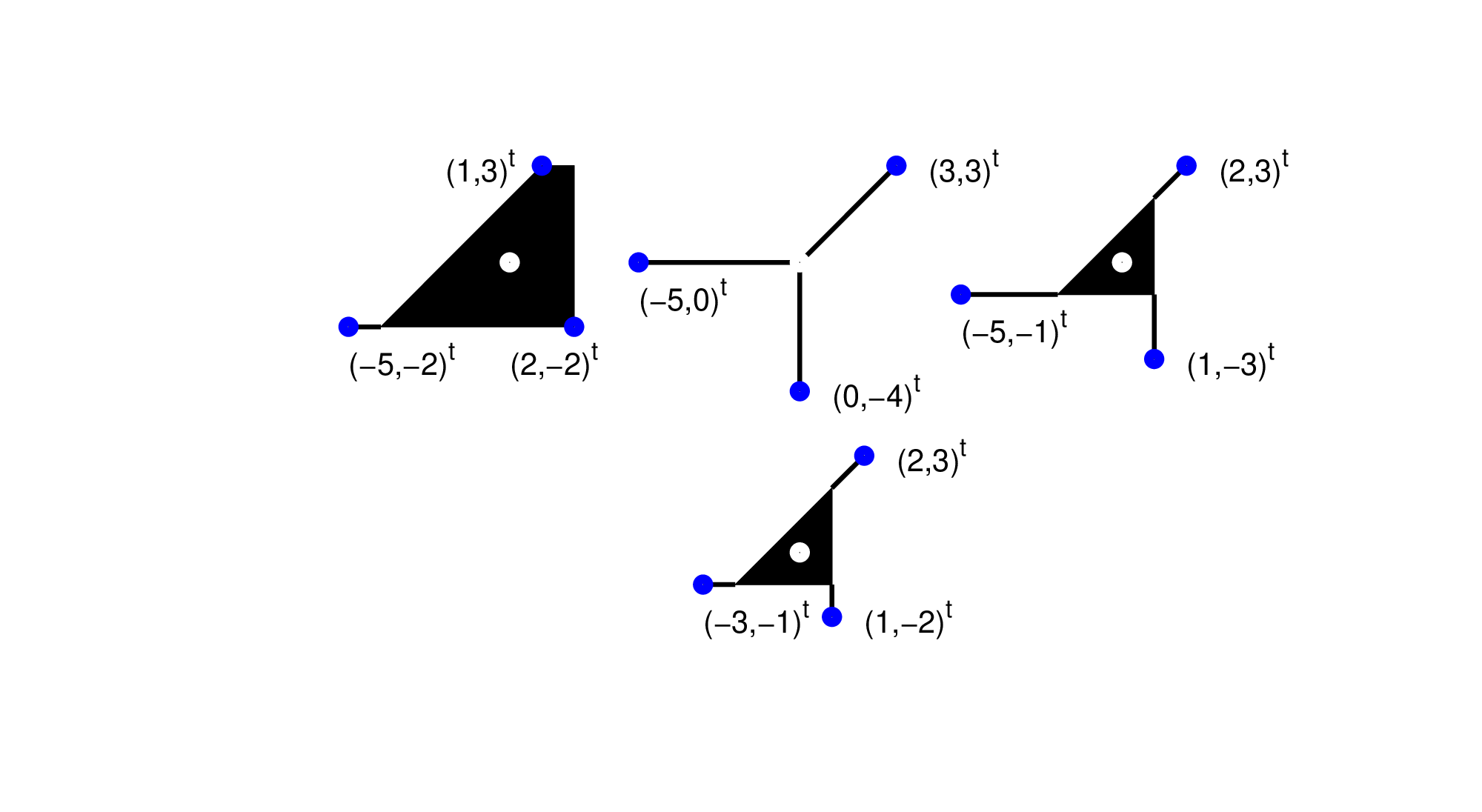}\\
  \caption{Top:  $\hipp\cap \spann A$ (left), $\hipp\cap \spann P(-p,0)$ (center) and  $\hipp\cap\spann B$ (right), for $p=(4,3,5)$. Bottom: $\hipp\cap \spann C$, with $C=AB=BA$. In each case, the zero vector is marked in white and generators  are represented in blue. The matrices $A,B$ and $C$ are perturbations of $P(-p,0)$.}\label{fig_01}
\end{figure}

\section{Geometry}\label{sec:geom}
Let $A,B\in\norma n$ be real.
Here we study the role played by the geometry of the complexes $\spann A$ and $\spann B$ in order to have $AB=BA$. To do so, we bear in mind how the maps $f_A$ and $f_B$ act, where $f_A:\overline{\R}^n \rightarrow \overline{\R}^n$ transforms a column vector $X$ into the product $AX$. For $n=3$, $f_A$ is described in detail in see \cite{Puente_lin}; see also  \cite{Rincon}.

\m
Before, we have met two instances where the geometry explains why $AB=BA$. Namely, in remarks after propositions \ref{prop:neighI} and \ref{prop:neigh0}. In the first (resp. second) case we have $AB=BA=A$ (resp. $AB=BA=B$) because $\spann B$ is much larger (resp. smaller) than $\spann A$.

\m
 More generally, we explore the relationship among the sets $\spann A$, $\spann B$, $\spann (AB)$ and $\spann (BA)$ when commutativity is present or absent. In general, we have $\spann (AB)\subseteq \spann A$ and  $\spann (BA)\subseteq \spann B$. In particular,  if $AB=BA$ then  $\spann (AB)\subseteq \spann A\cap \spann B$.

\begin{prop}\label{prop:spans}
Let $A,B\in\norma n$. If $A\le B=B^2$ and $A$ is real, then $A\in\Omega^B(B)$  and $\spann A\supseteq \spann B$.
\end{prop}

\begin{proof}
By normality, we have $I\le A\le B\le0$ and left or right tropical multiplication by any  matrix is monotonic. Therefore, $B\le AB\le B^2=B$ and $B\le BA\le B^2=B$, whence $AB=BA=B$ and $A\in\Omega^B(B)$.
Moreover, whatever the matrices $A$ and $B$ may be, we have $\spann A\supseteq \spann (AB)$ and, in our case, $\spann (AB)=\spann B$.
\end{proof}

The hypothesis $B=B^2$ cannot be removed in the previous proposition, as the following example shows.

\begin{ex}\label{ex:ultimo}
Consider
\begin{equation*}
A=\left[\begin{array}{rrr}
0&-1&-3\\0&0&-4\\0&0&0\\
\end{array}\right] \le B=\left[\begin{array}{rrr}
0&-1&-2\\0&0&-4\\0&0&0\\
\end{array}\right],
\end{equation*}
then
\begin{equation*}
AB=\left[\begin{array}{rrr}
0&-1&-2\\0&0&-2\\0&0&0\\
\end{array}\right] \neq BA=\left[\begin{array}{rrr}
0&-1&-2\\0&0&-3\\0&0&0\\
\end{array}\right],
\end{equation*}
and $\spann A\not\supseteq \spann B$;
 see figure \ref{fig_02}.
\end{ex}

\begin{figure}[h]
 \centering
  \includegraphics[keepaspectratio,width=13cm]{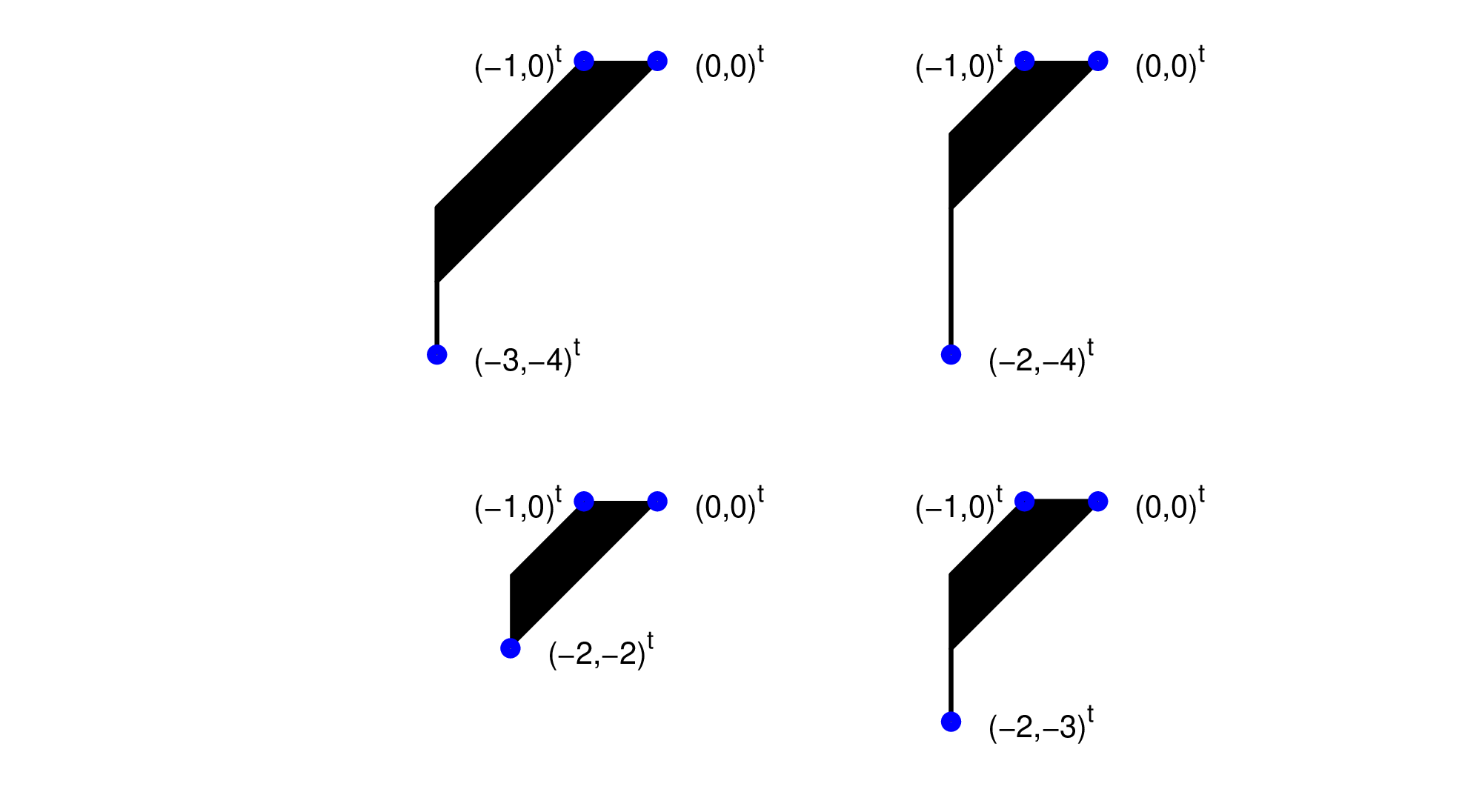}\\
  \caption{Top: $\{x_3=0\}\cap\spann A$ (left), $\{x_3=0\}\cap\spann B$ (right). Bottom: $\{x_3=0\}\cap\spann (AB)$ (left), $\{x_3=0\}\cap\spann (BA)$ (right). In this case, $\spann A\cap\spann B=\spann (AB)$. Generators  are represented in blue.}
  \label{fig_02}
\end{figure}

Below we explore the properties of the matrices $\underline{A}, A$ and $\overline{A}$ and of the corresponding polyhedral complexes.

\setcounter{thm}{10}
\begin{ex}\label{ex:de_nuevo}
(Continued) By proposition \ref{prop:ineq_bars}, we have
\begin{equation*}
\underline{B}=\left[\begin{array}{rrr}
0&-3&-3\\
-5&0&-6\\
-5&-2&0
\end{array}\right]\le
B=\left[\begin{array}{rrr}
0&-3&-1\\
-4&0&-6\\
-5&0&0
\end{array}\right]\le
\overline{B}=\left[\begin{array}{rrr}
0&-1&-1\\ -4&0&-5\\ -4&0&0 \end{array}\right]
\end{equation*}
and we can easily check, in this case, that
$$\spann \underline{B}\supseteq \spann B\supseteq \spann \overline{B}.$$
See figure \ref{fig_03}, where we are using the matrices
 \begin{equation*}
\underline{B}_0=\left[\begin{array}{rrr}
5&-1&-3\\
0&2&-6\\
0&0&0
\end{array}\right],
B_0=\left[\begin{array}{rrr}
5&-3&-1\\
1&0&-6\\
0&0&0
\end{array}\right],
\overline{B}_0=\left[\begin{array}{rrr}
4&-1&-1\\
0&0&-5\\
0&0&0
\end{array}\right],
\end{equation*}
as defined in p. \pageref{not:sub_cero}.
Notice that $\{x_3=0\}\cap \spann B$ is the union  of one closed 2--dimensional cell (called \emph{soma}) and three closed 1--dimensional cells (called \emph{antennas}); see \cite{Puente_lin} for the definition of soma, antennas and co--antennas (with a slightly different notation and language). In figure 3, bottom, we can see $\{x_3=0\}\cap \spann B$ together with its co--antennas.

In this example,
$$\overline{B}=B^*$$
and the matrix  $\underline{B}$ is idempotent.
Therefore, the sets $\spann \underline{B}$ and $\spann \overline{B}$ are classically convex, 
and so are the sections $\{x_3=0\}\cap\spann \underline{B}$ and $\{x_3=0\}\cap\spann \overline{B}$.

Consider $\HH$, the classical convex hull  of $\{x_3=0\}\cap\spann B$: its the vertices   are $(5,0)^t, (5,1)^t, (-2,1)^t, (-3,0)^t, (-3,-6)^t$ and $(-1,-6)^t$,  going counterclockwise.
Notice that $\{x_3=0\}\cap\spann \underline{B}$ is \emph{strictly larger} than $\HH$.
Actually, $\{x_3=0\}\cap\spann \underline{B}$ is the  convex hull of the \emph{union} of $\{x_3=0\}\cap\spann B$ and the \emph{co--antennas} of it. 
On the other hand, $\{x_3=0\}\cap\spann \overline{B}$ is the \emph{soma} of $\{x_3=0\}\cap\spann B$, i.e., it is the maximal convex set contained there.\qed
\end{ex}
\setcounter{thm}{25}

\begin{figure}[h]
 \centering
  \includegraphics[keepaspectratio,width=15cm]{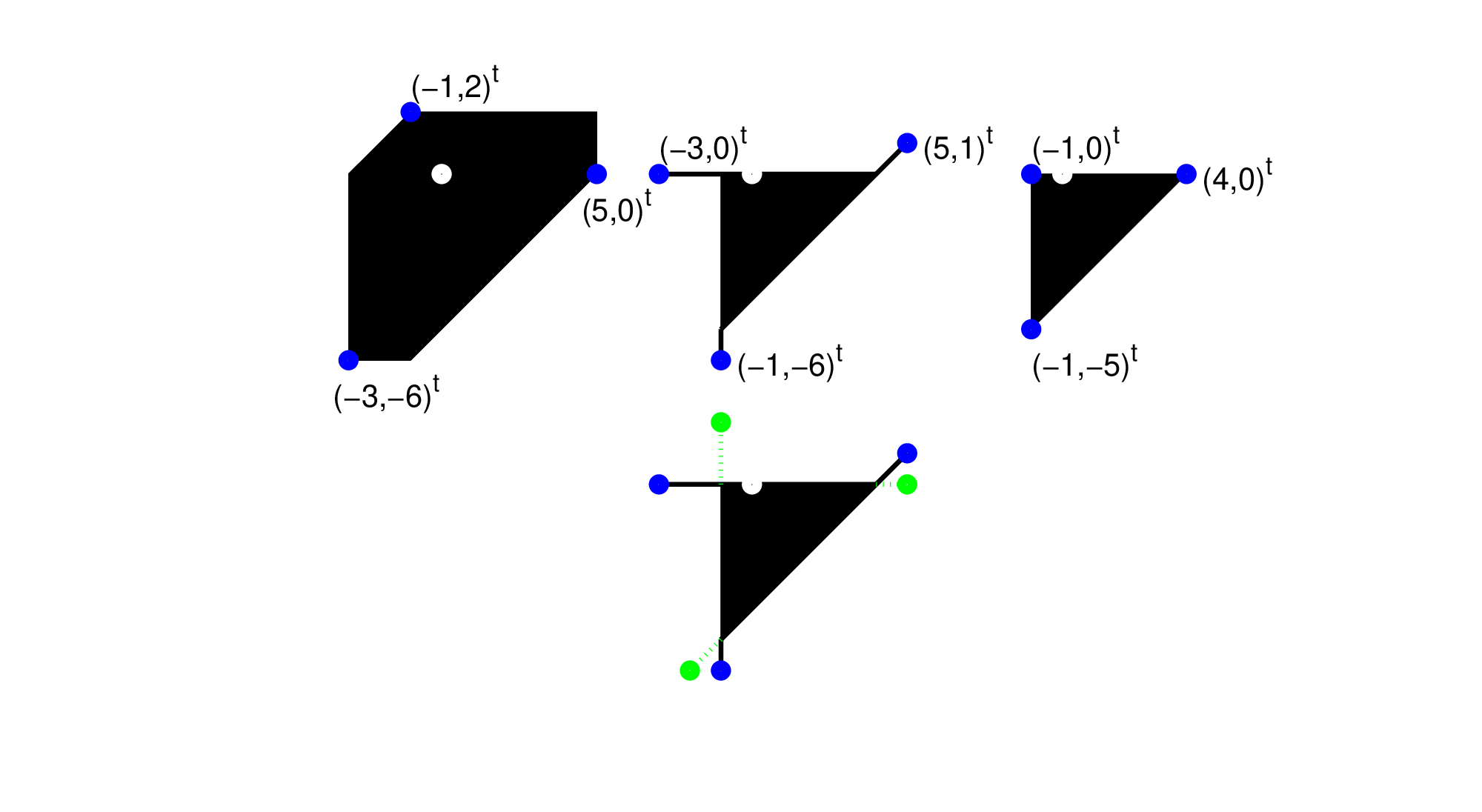}\\
  \caption{Top: $\{x_3=0\} \cap\spann \underline{B}$ (left), $\{x_3=0\} \cap\spann B$ (center)  and $\{x_3=0\} \cap\spann \overline{B}$ (right), for $B$ in (\ref{eqn:3x3}). In each case, the zero vector is marked in white, and generators (i.e., columns of the corresponding matrix $\underline{B}_0$, $B_0$ and $\overline{B}_0$) are represented in blue. The hyperplane section $\{x_3=0\}\cap \spann B$ has three antennas. Bottom: $\{x_3=0\}\cap \spann B$ is represented together with its co--antennas, which appear dotted  in green. The convex hull of the bottom figure is the top left one.}
  \label{fig_03}
\end{figure}
We wonder whether the statements in the former example are true for any real $B\in \norma n$. This is an open QUESTION. \label{quest:generalize}






\begin{thebibliography}{11}
\newcounter{bi}
\setcounter{bi}{-1}
\stepcounter{bi}

\addtocounter{bi}{1}\bibitem[\thebi]{Akian_HB} M. Akian,  R. Bapat and S. Gaubert, \emph{Max--plus algebra}, chapter 25 in \emph{Handbook of linear algebra}, L. Hobgen (ed.) Chapman and Hall, 2007.
%
\addtocounter{bi}{1}\bibitem[\thebi]{Baccelli} F.L. Baccelli, G. Cohen, G.J. Olsder and J.P. Quadrat, \emph{ Syncronization and linearity}, John Wiley; Chichester; New York, 1992.
%
\addtocounter{bi}{1}\bibitem[\thebi]{Brugalle_fran} E. Brugall\'{e}, \emph{Un peu de g\'{e}om\'{e}trie tropicale}, Quadrature, \textbf{74}, (2009), 10--22.

\addtocounter{bi}{1}\bibitem[\thebi]{Brugalle_engl} E. Brugall\'{e}, \emph{Some aspects of tropical geometry}, Newsletter of the European Mathematical Society, \textbf{83}, (2012), 23--28.


\addtocounter{bi}{1}\bibitem[\thebi]{Butkovic} P. Butkovi\v{c}, \emph{Max--algebra: the linear algebra of combinatorics?}, Linear Algebra Appl. \textbf{367}, (2003), 313--335.
%
\addtocounter{bi}{1}\bibitem[\thebi]{Butkovic_S} P. Butkovi\v{c}, \emph{Simple image set of $(\max,+)$ linear mappings}, Discrete Appl. Math. \textbf{105}, (2000), 73--86.

\addtocounter{bi}{1}\bibitem[\thebi]{Butkovic_libro} P. Butkovi\v{c}, \emph{Max--plus linear systems: theory and algorithms},  2010, Springer.

%
\addtocounter{bi}{1}\bibitem[\thebi]{Cuninghame} R. Cuninghame--Green, \emph{Minimax algebra}, LNEMS, \textbf{166}, Springer, 1970.

\addtocounter{bi}{1}\bibitem[\thebi]{Cuninghame_New} R.A.  Cuninghame--Green, \emph{Minimax algebra and applications}, in \emph{Adv. Imag. Electr. Phys.}, \textbf{90}, P. Hawkes, (ed.), Academic Press, 1--121, 1995.
%
\addtocounter{bi}{1}\bibitem[\thebi]{Develin_Sturm} M. Develin, B. Sturmfels,
\emph{Tropical convexity}, \emph{Doc. Math.} \textbf{9},  (2004) 1--27;  Erratum in Doc. Math. \textbf{9} (electronic),
(2004) 205--206.

%
\addtocounter{bi}{1}\bibitem[\thebi]{Gathmann}  A. Gathmann, \emph{Tropical algebraic
geometry}, Jahresbericht der DMV, \textbf{108}, n.1,
(2006), 3--32.
%
%
%
%


\addtocounter{bi}{1}\bibitem[\thebi]{Franceses}  I. Itenberg, E. Brugall\'{e}, B. Tessier, \emph{Géométrie tropicale}, Editions de l'\'{E}cole Polythecnique, 2008.
\addtocounter{bi}{1}\bibitem[\thebi]{Rusos}  I. Itenberg, G. Mikhalkin and E. Shustin, \emph{Tropical algebraic geometry},
Birkh{\"a}user,  2007.

%

\addtocounter{bi}{1}\bibitem[\thebi]{Johns_Kambi_Idempot}  M. Johnson and M. Kambites, \emph{Idempotent tropical matrices and finite metric spaces}, to appear in Adv. Geom.; arXiv: 1203.2480, 2012.

\addtocounter{bi}{1}\bibitem[\thebi]{Izha_Johns_Kambi_Proj}  Z. Izhakian, M. Johnson and M. Kambites, \emph{Pure dimension and projectivity of tropical politopes}, arXiv: 1106.4525, 2012.

\addtocounter{bi}{1}\bibitem[\thebi]{Izha_Johns_Kambi_Groups}  Z. Izhakian, M. Johnson and M. Kambites, \emph{Tropical matrix groups}, arXiv: 1203.2449, 2012.

\addtocounter{bi}{1}\bibitem[\thebi]{Jimenez_P} A. Jim\'{e}nez and M.J. de la Puente, \emph{Characterizing the convexity of the $n$--dimensional tropical simplex and  the six maximal convex classes in $\R^3$}, arXiv: 1205.4162, 2012.
%
\addtocounter{bi}{1}\bibitem[\thebi]{Joswig_Kulas} M. Joswig and K. Kulas, \emph{Tropical and ordinary convexity combined}, Adv. Geom. \textbf{10}, (2010)
333-352.

\addtocounter{bi}{1}\bibitem[\thebi]{Joswig_Sturm_Yu} M. Joswig, B. Sturmfels and J. Yu, \emph{Affine buildings and tropical convexity}, Albanian J. Math. \textbf{1}, n.4, (2007) 187--211.

\addtocounter{bi}{1}\bibitem[\thebi]{Katz_Schn_Sergeev} R. Katz, H. Schneider and S. Sergeev, \emph{On commuting matrices in max algebra and in classical nonegative algebra},  Linear Algebra Appl. \textbf{436}, (2012), 276--292.



\addtocounter{bi}{1}\bibitem[\thebi]{Kuhn}  H.W. Kuhn, \emph{The Hungarian method for the assignment problem}, Naval Res. Logist. \textbf{2}, (1955), 83--97.

\addtocounter{bi}{1}\bibitem[\thebi]{Lam_Postnikov}  T. Lam and A. Postnikov, \emph{
Alcoved polytopes I},  Discrete Comput. Geom.,   \textbf{38} n.3,  (2007)  453-478.

\addtocounter{bi}{1}\bibitem[\thebi]{Lam_PostnikovII}  T. Lam and A. Postnikov, \emph{
Alcoved polytopes II},  arXiv:1202.4015v1 (2012).

%
%
\addtocounter{bi}{1}\bibitem[\thebi]{Litvinov_ed}  G.L. Litvinov, V.P. Maslov, (eds.) \emph{
Idempotent mathematics and mathematical physics},  Proceedings Vienna 2003, American
Mathematical Society, Contemp. Math. \textbf{377}, (2005).

%


%
\addtocounter{bi}{1}\bibitem[\thebi]{Mikhalkin_W}  G. Mikhalkin, \emph{What is a tropical curve?}, Notices AMS, April 2007, 511--513.

\addtocounter{bi}{1}\bibitem[\thebi]{Papa}  C.H. Papadimitriou and K. Steiglitz, \emph{Combinatorial optimization: algorithms and complexity}, Prentice Hall,   1982 and corrected unabrideged republication by Dover, 1998.

\addtocounter{bi}{1}\bibitem[\thebi]{Prasolov}  V.V. Prasolov, \emph{Problems and theorems in linear algebra}, AMS,   1994.

\addtocounter{bi}{1}\bibitem[\thebi]{Puente_lin} M. J. de la Puente, \emph{Tropical linear maps on the plane},
Linear Algebra Appl. \textbf{435}, n.7, (2011) 1681--1710.

\addtocounter{bi}{1}\bibitem[\thebi]{Puente_kleene} M. J. de la Puente, \emph{On tropical Kleene star matrices and alcoved polytopes}, Kybernetika, \textbf{49}, n.6, (2013) 897--910.

\addtocounter{bi}{1}\bibitem[\thebi]{Richter}  J. Richter--Gebert, B. Sturmfels, T.
Theobald, \emph{First steps in tropical geometry}, in
\cite{Litvinov_ed},  289--317.
\addtocounter{bi}{1}\bibitem[\thebi]{Rincon}  F. Rinc\'{o}n,
 \emph{Local tropical linear spaces}, Discrete Comput. Geom. \textbf{50}, (2013), 700--713.
%

\addtocounter{bi}{1}\bibitem[\thebi]{Sergeev_def} S. Sergeev, \emph{Max--plus definite matrix closures and their eigenspaces,} Linear
Algebra Appl. \textbf{421}, (2007) 182--201.


\addtocounter{bi}{1}\bibitem[\thebi]{Sergeev_S_B}  S. Sergeev, H. Scheneider and P. Butkovi\v{c},   \emph{On visualization, subeigenvectors and Kleene stars in max algebra}, Linear Algebra Appl. \textbf{431}, 2395--2406, (2009).
%
\addtocounter{bi}{1}\bibitem[\thebi]{Speyer_Sturm_Mag}  D. Speyer, B. Sturmfels,   \emph{Tropical mathematics},
Math. Mag.  \textbf{82}, n.3, (2009) 163--173.
\addtocounter{bi}{1}\bibitem[\thebi]{Wagneur_M} E. Wagneur, \emph{Modulo\"{\i}ds and
pseudomodules. Dimension theory}, Discr. Math. \textbf{ 98}
(1991) 57--73.

\addtocounter{bi}{1}\bibitem[\thebi]{Werner_Yu} A. Werner and J. Yu, \emph{Symmetric alcoved polytopes}, arXiv: 1201.4378v1 (2012).

\addtocounter{bi}{1}\bibitem[\thebi]{Yoeli} M. Yoeli, \emph{A note on a generalization of boolean matrix theory}, Amer. Math. Monthly \textbf{68}, n.6,
(1961) 552--557.

\addtocounter{bi}{1}\bibitem[\thebi]{Zimmermann_K} K. Zimmermann, \emph{Extrem\'{a}ln\'{\i} algebra},  V\'{y}zkumn\'{a} publikace ekonomicko--matematick\'{e} laborato\v{r}e p\v{r}i ekonomick\'{e}m \'{u}stav\'{e} \v{C}SAV, \textbf{46}, Prague, 1976, in Czech.



\end{thebibliography}
\end{document}